\documentclass[10pt,reqno]{amsart}
\usepackage{latexsym}
\usepackage{graphicx}
\usepackage{fancyhdr,amssymb}

\theoremstyle{plain}
\numberwithin{equation}{section}
\newtheorem{thm}{Theorem}[section]
\newtheorem{theorem}[thm]{Theorem}
\newtheorem{lemma}[thm]{Lemma}

\newtheorem{corollary}[thm]{Corollary}

\def\cp{{ \,\Box\,}}

\theoremstyle{definition}

\begin{document}
\fancyhead{}
\renewcommand{\headrulewidth}{0pt}
\fancyfoot{}
\fancyfoot[LE,RO]{\medskip \thepage}
\fancyfoot[LO]{\medskip MISSOURI J.~OF MATH.~SCI., SPRING 2021}
\fancyfoot[RE]{\medskip MISSOURI J.~OF MATH.~SCI., VOL.~33, NO.~1}

\setcounter{page}{1}

\title[On dispersability of some products of cycles]{On dispersability of some products of cycles}
\author{Samuel S. Joslin}
\address{Department of Mathematics and Statistics\\
                Georgetown University}
\email{ssj34@georgetown.edu}
\author{Paul C. Kainen}
\address{Department of Mathematics and Statistics\\
                Georgetown University}
\email{kainen@georgetown.edu}
\author{Shannon Overbay}
\address{Department of Mathematics\\
                Gonzaga University}
\email{overbay@gonzaga.edu}

\begin{abstract}
We show that the matching book thickness of the Cartesian product of two odd-length cycle-graphs is five if at least one of the cycles has length 3 or 5.
\end{abstract}

\maketitle

\section{Introduction}
\noindent The chromatic index $\chi'(G)$ of a graph $G$ is the least number of colors needed in a proper coloring of the edges, so that no two same-color edges share a common endpoint.  Unlike chromatic number, which can be as small as 2 and as large as the number of vertices, chromatic index has only two possible values, $\Delta(G)$ or $1+\Delta(G)$, according to whether $G$ is of Vizing class I or II \cite[p. 133]{harary}, where $\Delta(G)$ denotes maximum degree of $G$.

Surprisingly, for a number of graph families, it is possible to enrich the edge-colorings so that they are compatible with a geometric layout of the graph, 
yet not requiring additional colors \cite{abgkp2018,bk79,jko-2021b,pck-90,pck-bica,ko-2020,so-thesis,shao-et-al,shao-3,shao-2}.  For instance,  we showed in \cite{pck-bica} that the Cartesian product of cycles has a layout and compatible coloring with $\chi'(G)$ colors if at most one of the cycles is of odd length. 
Also, cubic planar bipartite,  regular complete bipartite, and hypercube graphs all have such enriched $\chi'$-edge-colorings \cite{abgkp2018,ko-2020,bk79,so-thesis}.
This is not true for all graphs; the complete graph $K_{2r}$, $r \geq 2$, is class I but requires one extra color to be compatible with its layout (which is unique as the graph is complete). 

For a convex geometric layout of a graph $G$, the chromatic number of the intersection graph of the open edges of the drawing is the {\it book thickness} \cite{bk79}. If we take {\it closed} edges in forming the intersection graph, then the chromatic number is the {\bf matching book thickness} \cite{pck-bica} of the graph-layout pair. The matching book thickness $mbt(G)$ of $G$ is the minimum over all its layouts.

\par
\ \\
\noindent\rule{0.84in}{0.4pt} \par
\medskip
\indent\indent {\fontsize{8pt}{9pt} \selectfont DOI: 10.35834/YYYY/VVNNPPP \par}
\indent\indent {\fontsize{8pt}{9pt} \selectfont MSC2020: 54A40 \par}
\indent\indent {\fontsize{8pt}{9pt} \selectfont Key words and phrases: matching book thickness; chromatic index; dispersable graph\par}

\thispagestyle{fancy}
\vfil\eject
\fancyhead{}
\fancyhead[CO]{\hfill ON DISPERSABILITY OF SOME PRODUCT OF CYCLES}
\fancyhead[CE]{S. S. JOSLIN, P. C. KAINEN, AND S. OVERBAY  \hfill}
\renewcommand{\headrulewidth}{0pt}

By definition, $mbt(G) \geq \chi'(G)$. Call $G$ {\bf dispersable} \cite{bk79}
if $mbt(G) = \Delta(G)$ and {\bf nearly dispersable} \cite[p 88]{so-thesis} if $mbt(G) = 1 + \Delta(G)$. In \cite[p 87]{so-thesis}, Overbay proved that {\it if $G$ is regular and dispersable, then it must be bipartite}.  Thus, every regular non-bipartite Vizing class I graph has matching book thickness at least $1 + \Delta(G)$.  No currently known graph (of either Vizing class) has $mbt(G) \geq 2 + \Delta(G)$ though it can be shown \cite{bmw2006} that such graphs exist when $\Delta(G) \geq 9$ and the order $|G|$ is sufficiently large.

The question of determining which graphs are dispersable or nearly dispersable has been considered since the conjecture of \cite{bk79} that regular bipartite graphs are dispersable. Alam-et-al in \cite{abgkp2018} disproved that conjecture by finding nearly dispersable regular bipartite graphs but showed that 3-connected cubic planar bipartite  graphs {\it are} dispersable.  In \cite{ko-2020}, we extended the result of \cite{abgkp2018}, by showing  the 3-connected case implies the general case as conjectured.

In this paper, graphs are simple and ``$\cp$'' denotes Cartesian product; $C_n$ is an $n$-cycle.
We continue \cite{pck-bica} by studying $mbt(C_m \cp C_n)$.  A non-self-crossing cycle is dispersable if it is of even length and nearly dispersable if it is of odd length. But cycles can have crossings if the crossing edge-pairs have distinct colors.  In \cite{pck-bica} it was proved that the product of two even cycles is dispersable, and the product of an even and an odd cycle is nearly dispersable.  Here we show $C_m \cp C_n$ is nearly dispersable if $mn$ is odd and $\min(m,n) \leq 5$ as a first step towards our conjecture that the same conclusion holds for any product of two odd cycles, and hence for all cycle-pairs, odd or even. 

Other recent results include Shao-et-al \cite{shao-et-al} who show $K_n \cp C_m$ is nearly dispersable when $n,\, m \geq 3$, which implies Corollary \ref{co:C3} below.  The Halin trees with $\Delta \geq 4$ are also dispersable \cite{shao-3}.
Fiorini \cite{fiorini} proved that every outerplanar  graph (except an odd cycle) has chromatic index equal to its maximum degree, so outerplanar graphs (other than odd cycles) are dispersable; see also \cite{shao-2}. 

The paper is organized as follows: Section 2 gives nearly dispersable embeddings for $C_3 \cp C_3$ and $C_5 \cp C_5$.  
In Section 3, we prove that for suitable graphs $H$, there are certain properties of a matching book embedding of $H \cp C_s$  which enable its extension to a drawing of $H \cp C_{s+r}$, where $r$ is positive and  even, in such a way that the extended drawing maintains the required properties. 
The extension process is based on replication of a substructure we call a  ``seed.''  Such seeds do exist in the drawings of $C_3 \cp C_3$ and $C_5 \cp C_5$. 
We conclude in Section 4 with some remarks on further applications of our methods to graphs with a repetitive structure.

\section{Cartesian product and thickness}

\noindent Two distinct edges in a graph $G$ are said to be {\bf adjacent} if they share a unique vertex.  A {\bf layout} of a graph $G$ is a cyclic permutation $\omega$ of $V_ G$, two edges are called {\bf conflicting} (w.r.t. $\omega$) if their four endpoints occur in alternating order (equivalently, the corresponding chords cross). Let $[k]:=\{1,\ldots,k\}$.
A $k$-page {\bf matching book embedding mbe} is a triple $(G, \omega, c)$, where $\omega$ is a layout of $G$ and $c: E_G \to [k]$  is a surjection such that $c(e) \neq c(e')$ whenenver $e$ and $e'$ are adjacent or $e$ and $e'$ conflict with respect to $\omega$.  Thus, $mbt(G)$ is the least $k$ such that some $k$-page mbe  $(G,\omega,c)$ exists, while $mbt(G,\tau)$ is the least $k$ such that some $k$-page mbe  $(G,\tau,c)$ exists.

To show that $C_m \cp C_n$ is nearly dispersable when $m$ and $n$ are odd and $\min(m,n) \leq 5$,
we start with the cases $m=3=n$ and $m=5=n$.

\begin{lemma}
With the previous notation, $C_3 \cp C_3$ is nearly dispersable.
\label{lm:C3}
\end{lemma}

\begin {proof}
We show $C_3 \cp C_3$ is nearly dispersable by giving an explicit counter-clockwise ordering $\omega$ and edge coloring $c$. Let $G$ be $C_3 \cp C_3$. 
Order the vertices of $G$ counter-clockwise around a circle as follows: $$\omega = (1, 2, 3, 6, 5, 4, 7, 8, 9)$$ 

\noindent Partition the edges of $G$ into 5 pages using the following coloring:

\begin{enumerate}
\item Red: $1{-}2, \,3{-}9, \,5{-}6, \,8{-}7$ (medium width, solid)
\item Black: $ 1{-}3,\, 9{-}6,\, 8{-}5,\, 7{-}4$ (thin, dashed)
\item Green: $2{-}3,\, 6{-}4,\, 1{-}7, \,8{-}9$ (thick, dashed)
\item Blue: $2{-}8,\, 4{-}5$ narrow, irregular dashing pattern
\item Purple: $3{-}6,\, 2{-}5,\, 1{-}4, \,9{-}7$ (thin, solid)
\end{enumerate}
See Figure \ref{Figure 1}; same color/line-type coding is used throughout the paper.
\end{proof}

\begin{figure}[ht!]
\centering
\includegraphics[width=46mm]{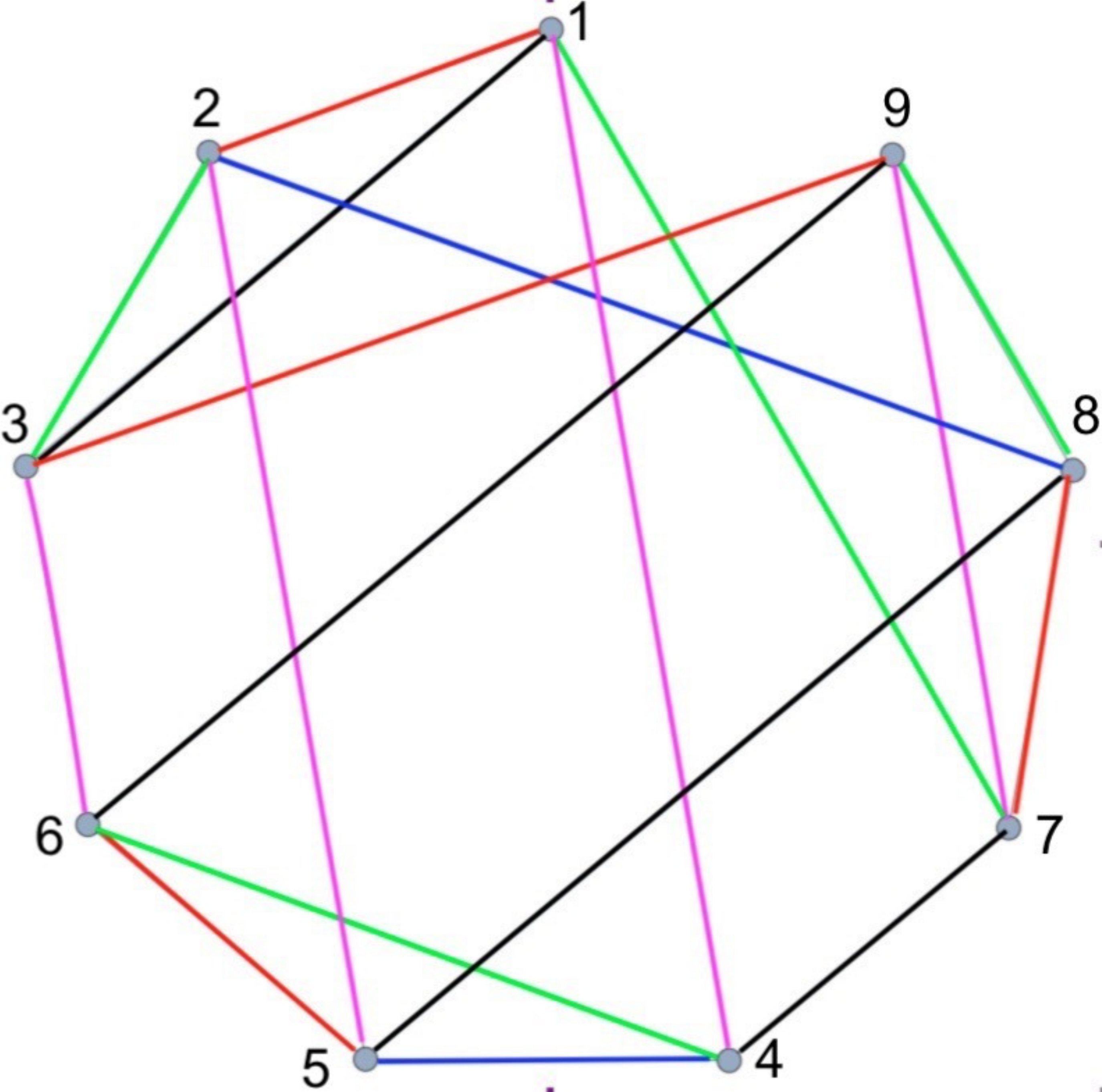}
\caption{A nearly dispersable embedding of $C_3 \cp C_3$ \label{Figure 1}}
\end{figure}

\begin{lemma}
With the previous notation, $C_5 \cp C_5$ is nearly dispersable.
\label{lm:C5}
\end{lemma}
\begin{proof}
We will show $C_5 \cp C_5$ is nearly dispersable by giving an explicit counter-clockwise ordering $\omega$ and edge coloring $c$. Let $G$ be $C_5 \cp C_5$. Order the vertices of G counter-clockwise around a circle as follows: 
$$\omega {=} 
(1,2,3,4,5,10, 9, 8, 7, 6, 11, 12, 13, 15, 14, 19, 20, 16, 17,18, 23, 22, 21, 25, 24)$$
\noindent Partition the edges of $G$ using the following coloring $c$: 
\begin{enumerate}
\item Purple: $1{-}2,\,3{-}4,\,9{-}14,\,8{-}7,\,11{-}15,\,12{-}13,\,24{-}19,\,25{-}20,\\\,21{-}16,\,22{-}17,\,23{-}18\,$
\item Blue: $2{-}3,\,1{-}5,\,10{-}9,\,7{-}6,\,11{-}12,\,13{-}18,\,15{-}20,\,14{-}19,\,16{-}17,\\\,21{-}22,\,23{-}24\,$
\item Red: $5{-}10,\,4{-}9,\,3{-}8,\,2{-}7,\,1{-}6,\,12{-}17,\,13{-}14,\,22{-}23,\,25{-}21\,$
\item Green: $24{-}25,\,1{-}21,\,2{-}22,\,3{-}23,\,4{-}5,\,10{-}15,\,8{-}13,\,7{-}12,\,6{-}11,\\\,19{-}18,\,20{-}16\,$
\item Black: $4{-}24,\,5{-}25,\,10{-}6,\,9{-}8,\,11{-}16,\,15{-}14,\,19{-}20,\,18{-}17\,$
\end{enumerate}
See Figure \ref{Figure 2}. Recall the line-type coding from Figure \ref{Figure 1}. \end{proof}

\begin{figure}[ht!]
\centering
\includegraphics[width=80mm]{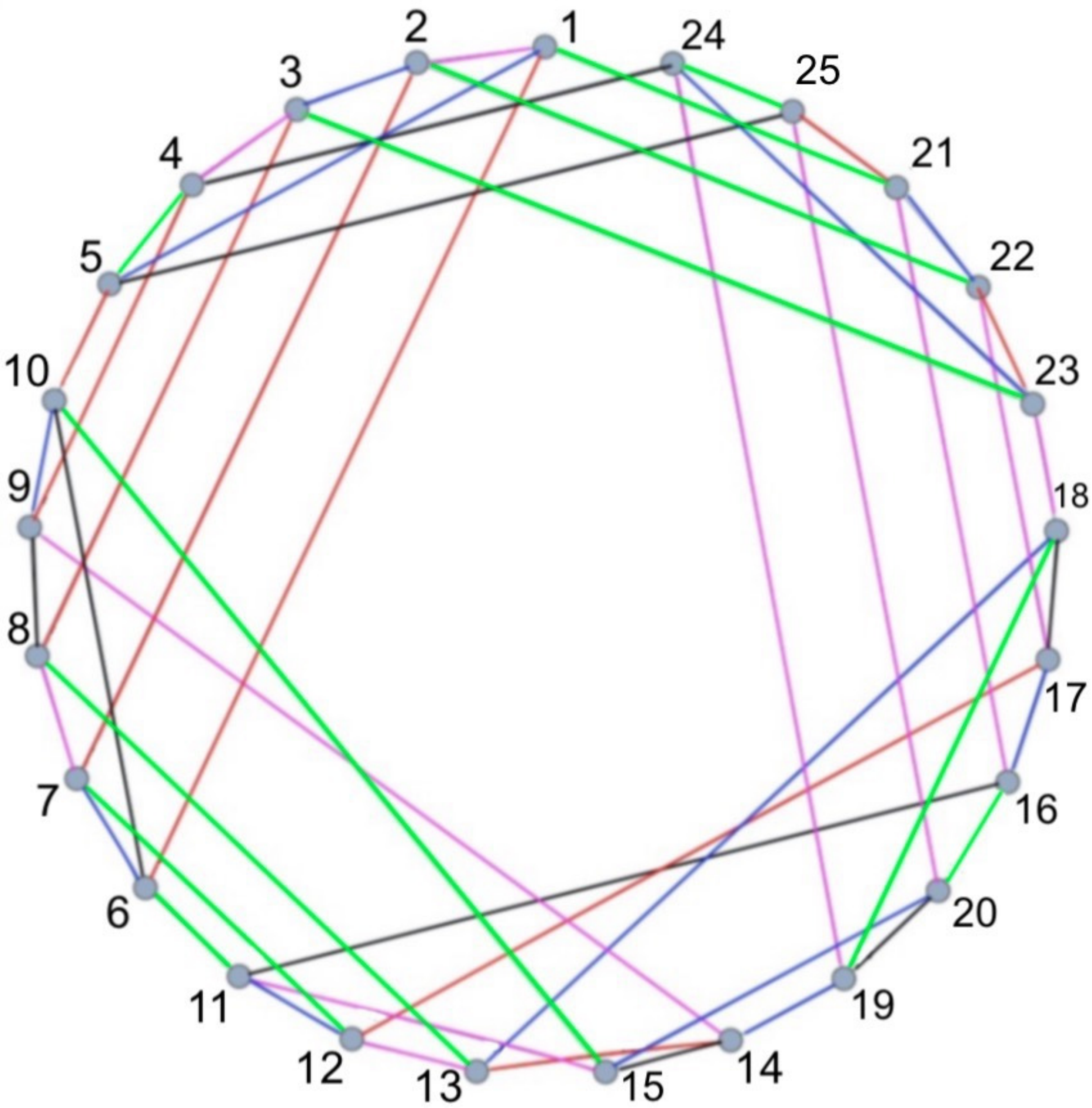}
\caption{ A nearly dispersable  embedding of $C_5 \cp C_5$ \label{Figure 2}}
\end{figure}

\section{Structure and extensions of the layouts}

\noindent The constructions in Lemmas \ref{lm:C3} and \ref{lm:C5} use layouts and edge-colorings for 
$C_3 \cp C_3$ and $C_5 \cp C_5$ which satisfy some extra properties that allow them to be extended to $C_3 \cp C_n$ and $C_5 \cp C_n$, $n\geq 3$ (or $\geq 5$) odd.  We describe these properties in a general context.

Let $s \geq 3$ and let $H$ be any $t$-regular nonbipartite graph, $t \geq 2$; later $H = C_3$ and $s=3$ or $H = C_5$  and $s=5$.  Put $G=H \cp C_s$.  A layout $\omega$ of $G$ will be called {\bf en bloc} if (1)
the set (or {\bf block}) of all vertices with a common second coordinate occur consecutively in $\omega$
and (2) the blocks corresponding to the vertices of $C_s$ occur in the natural counterclockwise ordering of $C_s$.  
The vertices in each block induce a subgraph of $H \cp C_s$ which is isomorphic to $H$ but the vertex orderings can be different from block to block.

By the definition of Cartesian product, each block is connected to the adjacent blocks before and after (in counterclockwise order)  by matchings (the set of $|H|$ independent edges in the product which correspond to the appropriate edge of $C_s$).  We call these two matchings the {\bf before} and {\bf after} matchings, and they can contain crossings (both intra- and  inter-block).

A matching book embedding of $H \cp C_s$ is {\bf extensible} if the layout is en bloc, there are $t+3$ pages (i.e., the embedding is nearly dispersable), and two additional conditions hold: (a) some block (the {\bf seed}) uses only $t+1$ edge-colors and (b) the two unused colors are {\bf separated} by the before and after matchings of the seed in the sense that neither matching contains edges of {\it both} unused colors.

For example, the block with vertices numbered $11,12,13,15,14$ in Figure  \ref{Figure 2} (at the bottom) uses colors purple, blue, red, and black so it fails (a).
In fact, this is the only block in either figure which fails to satisfy (a); all of the other blocks satisfy both (a) and (b) and hence are seeds.

However, it is possible for a block to satisfy (a) but not (b).  Indeed, a block whose layout is a $5$-cycle could be edge-colored $r, b, r, g$ in counterclockwise order along the circle and one inner edge colored $b$. The before-matching could be $p, k, k, k, k$, while the after-matching is $g, p, p, p, p$.  Condition (b) fails in the before-matching.  See Figure \ref{Figure 3}.

\begin{figure}[ht!]
\centering
\includegraphics[width=80mm]{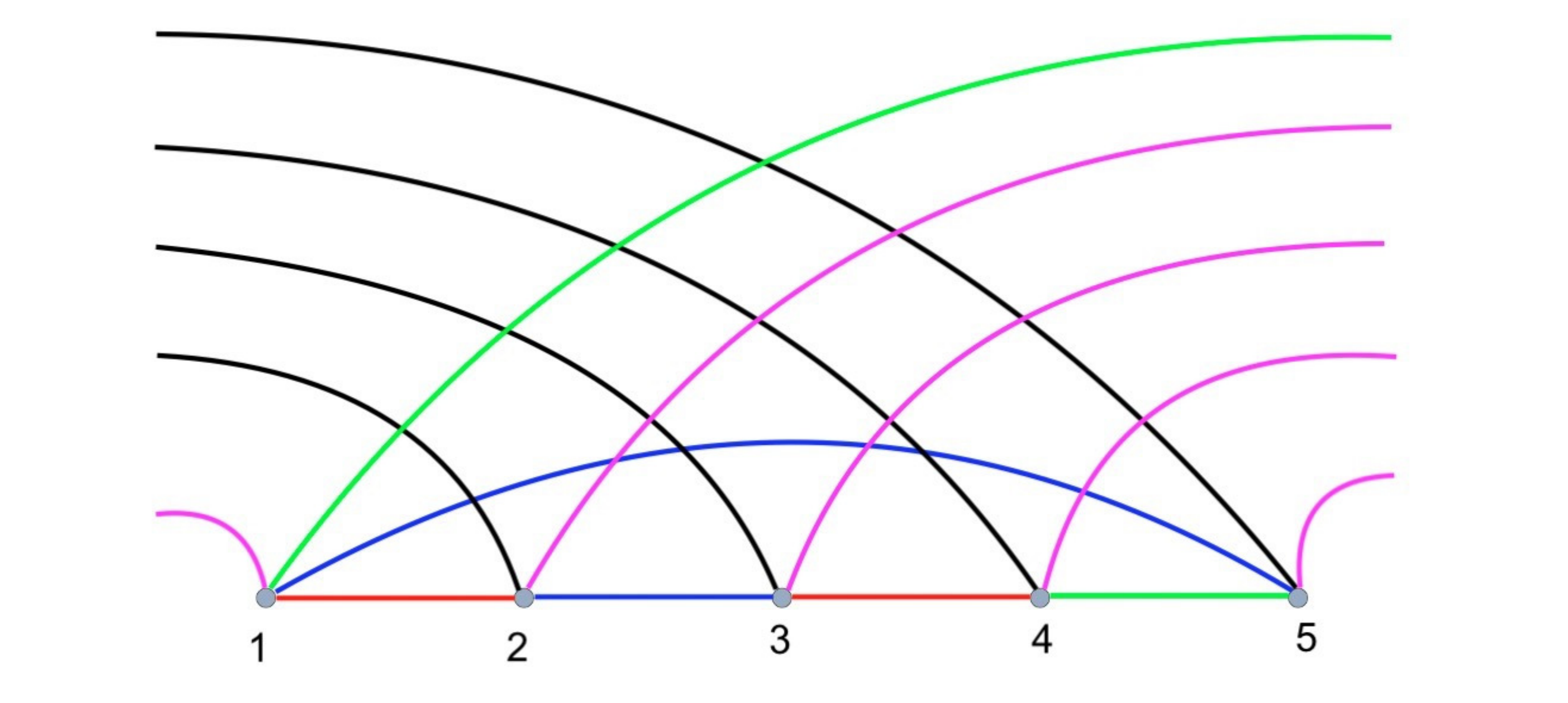}
\caption{A block that satisfies (a) but not (b)\label{Figure 3}}
\end{figure}

\begin{theorem}
Let $H$ be a $t$-regular, nonbipartite graph and let $s \geq 3$.  If $H \cp C_s$ has an extensible book embedding,  then {\rm for every even integer} $r>0$, $H \cp C_{s+r}$ also has an extensible book embedding which extends the given book embedding.
\end{theorem}
\begin{proof}
Take some seed $H_j := H \times \{j\}$ in the book embedding of $G := H \cp C_s$, where $1 \leq j \leq s$, and replace it by a sequence of length $r+1$, beginning (and ending) with an identical copy of the seed (and its edge-coloring), alternating with $H_j^{op}$
which is an order-reversed copy of $H_j$.  The iteratively reversed blocks are joined by monochromatic, noncrossing matchings which alternate with the two colors not used in $H_j$.  However, the before-matching of the first-counterclockwise copy of $H_j$ remains unchanged as also does the after-matching from the last-ccw copy of $H_j$.  This extends the layout and edge-coloring of $H \cp C_s$ to a layout and edge-coloring  of $H \cp C_{s+r}$ retaining extensibility.  See Figure \ref{Figure 4}, where $H = C_3, \;s=3, \;r=2$.
\end{proof}

\begin{figure}[ht!]
\centering
\includegraphics[width=80mm]{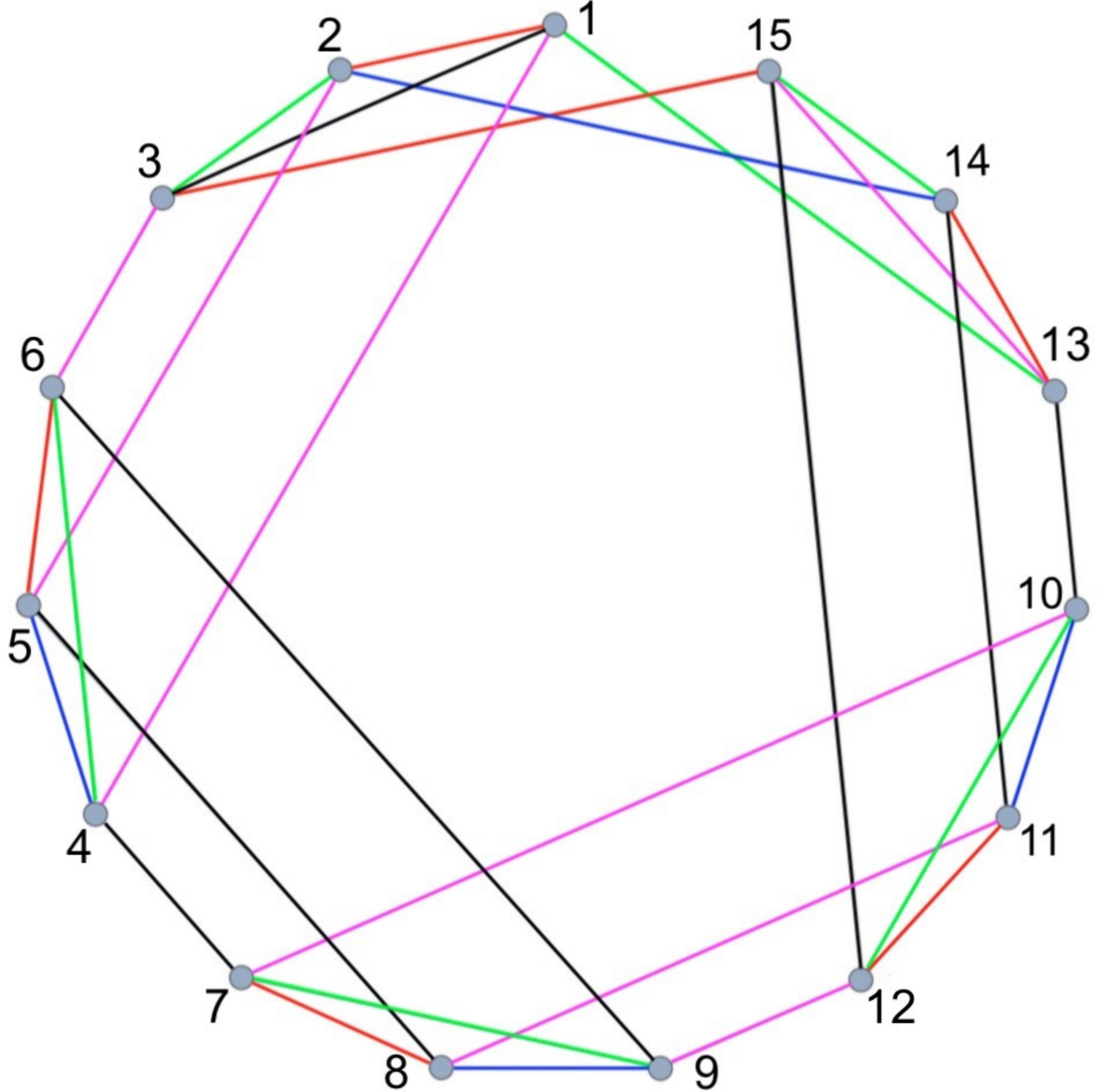}
\caption{A nearly dispersable  embedding of $C_3 \cp C_5$ \label{Figure 4}}
\end{figure}

As the layouts and edge-colorings in Lemmas 1 and 2 satisfy the conditions of Theorem 1, by \cite{pck-bica} for $n$ even, we have the following consequences.

\begin{corollary} 
If $n \geq 3$, then $C_3 \cp C_n$ is nearly disperable.
\label{co:C3}
\end{corollary}

\begin{corollary}
If $n\geq 3$,  then $C_5 \cp C_n$ is nearly dispersable.
\end{corollary}

\begin{corollary}
If $m=3$ or $5$, $n\geq 3$ odd,  and $H$ is bipartite and dispersable, then $C_m \cp C_n \cp H$ is nearly dispersable.
\end{corollary}

For Corollary 3.4 and in the introduction, we extended \cite{pck-bica} from pairwise to finite products: if $H$ is bipartite, then $mbt(G \cp H) \leq mbt(G) + mbt(H)$ for all graphs $G$ \cite[Thm. 4.3]{bk79}, \cite{pck-bica}, \cite[p 87]{so-thesis}, so the product of bipartite dispersable graphs is dispersable and the product of a bipartite dispersable graph and a nearly dispersable graph is nearly dispersable.

\section{Remarks}
\noindent We have given a five stack layout \cite{heath-et-al} of a toroidal grid (circumference 3 or 5) where at most one process starts or ends at a single time step in each stack.
The ideas in our construction,  en bloc layouts and replicating seeds satisfying internal and boundary conditions, may be useful in other types of polymer-like structures such as circulant graphs \cite{jko-2021b}.
If an extensible drawing of $C_7 \cp C_7$ can be found (this involves a finite search that should be implementable via computer), then by replicating the chosen seed, 
one would get a nearly dispersable  embedding of $C_7 \cp C_m$, for each odd $m \geq 7$.  

In particular, the extensibility technique described here allows a finite search, if it is successful, to determine an infinite family of answers.

\end{document}